\documentclass[12pt]{article}
\usepackage{mathrsfs}
\usepackage{amsfonts}
\usepackage{amssymb,amsmath,amsthm}
\usepackage{color}
\usepackage{verbatim} 

\usepackage{color} 

\everymath={\displaystyle}

\newcommand{\RR}{{\mathbb{R}}}

\newcommand{\eps}{\varepsilon}

\theoremstyle{thmstyleone}%
\newtheorem{theorem}{Theorem}
\newtheorem*{theorem*}{Theorem}
\newtheorem{proposition}[theorem]{Proposition}%
\newtheorem{lem}[theorem]{Lemma}%
\newtheorem{corollary}[theorem]{Corollary}%
\theoremstyle{thmstyletwo}%
\newtheorem{remark}{Remark}%

\theoremstyle{thmstylethree}%

\begin{document}

 \title{Symmetry Breaking in Biharmonic Equations with Weighted Exponential Nonlinearities}

\author{Marta Calanchi,  Cristina Tarsi $\thanks{{\small E-mail: marta.calanchi@unimi.it, cristina.tarsi@unimi.it }}$\\
	{\small Dipartimento di Matematica, Universit\`{a} degli Studi di Milano,}\\
	\small{Via C. Saldini, 50, 20133 Milan, Italy}
}
\date{}

\maketitle

\abstract{
We investigate a class of fourth-order elliptic problems involving exponential-type nonlinearities and spatial weights of H\'enon type. Motivated by the symmetry-breaking phenomena observed in semilinear second-order problems -- such as those governed by the H\'enon equation -- we consider weighted functionals of the form
\begin{equation*}
F_m(u) = \int_B |x|^\alpha \left( e^{\sigma |u|^2} - \sum_{k=0}^m \frac{\sigma^k}{k!} |u|^{2k} \right) dx,
\end{equation*}
defined on the unit ball \( B \subset \mathbb{R}^4 \), where $m\in \mathbb N_0$ \( \alpha > 0 \),   \( \sigma>0\) are suitable parameters. We first establish an Adams-type inequality with weight, characterizing the sharp threshold for the boundedness of \( F \) on the unit sphere of the biharmonic Sobolev space. Then, we prove that for large values of the weight exponent \( \alpha \), radial symmetry of maximizers is broken.
These results extend classical findings in the second-order setting 
(e.g., Trudinger--Moser-type functionals and the weighted H\'enon equation)
 to the biharmonic context and offer new insights into the interplay between weights, nonlinearity, and symmetry in higher-order PDEs.
 }

\vskip2mm
\par
{\bf Keywords:} H\'enon type problem, symmetric breaking, biharmonic operator Adams' Moser Trudinger inequalities.

\vskip2mm
\par
{\bf MSC(2010): } 35J60,35J40

\section{Introduction}
In this work we investigate  the phenomenon of symmetry breaking for max\-imization problems involving fourth--order elliptic operators with critical exponen\-tial growth
 of Adams type under the presence of a H\'enon weight in dimension four.

Our aim is twofold: on the one hand, we wish to extend to the biharmonic setting
several classical ideas originating from the work of H\'enon and Ni for second--order
operators; on the other hand, we identify in a precise and quantitative manner the
mechanisms which lead to the loss of radial symmetry for extremal functions, par-
alleling the foundational contributions of Smets--Su--Willem and Calanchi--Terraneo
in lower--order cases.

The starting point of our investigation is the celebrated H\'enon problem
\begin{equation}\label{eq1}
\left\{\begin{array}{rll}
-\Delta u & =|x|^{\alpha}u^{p-1}\ \quad  &\hbox{in}\ B\\
u&>0\\
u&=0\quad\quad\quad & \hbox{on}\ B,
\end{array}
\right.
\end{equation}
(here $B\subset \RR^N$ is a ball centered at the origin, $N\geq 2, \alpha >0$ and $p>2$)  
which arose from H\'enon's 1973 mathematical model~\cite{h} describing the
distribution of spherical stellar clusters subjected to anisotropic forces.
A distinctive feature of the model lies in the power--type radial weight
\( |x|^\alpha \), which shifts the concentration of solutions away from the
origin, producing in turn profound changes in the analytic structure of the
associated variational problem.

\medskip

Ni's seminal paper~\cite{ni} opened the modern analysis of the H\'enon
problem.  
Working with the energy functional in \(H_0^1(B)\), he proved that the
presence of the weight significantly enlarges the range of
exponents \(p\) for which radial solutions exist, namely
\[
p\in\Bigl(2,\, 2^* + \tfrac{2\alpha}{N-2} \Bigr), \qquad
2^*=\tfrac{2N}{N-2},
\]
and also showed, by means of the Pohozaev identity, that this range is sharp.
Ni's insight revealed a fundamental mechanism: the H\'enon weight interacts
with scaling in such a way as to effectively ``shift'' the critical exponent
of the embedding.  
This phenomenon, though simple to state, has since been shown to have
farreaching analogues in a wide spectrum of elliptic problems.

%

\smallskip

Nonradial solutions, however, appear beyond Ni's radial framework.  
\\
A first numerical indication came from Chen--Ni--Zhou~\cite{cnz}, who produced
explicit nonradial ground states in dimension two.  
The complete picture was given by Smets, Su, and Willem~\cite{ssw}, who
studied the minimizers (ground state solutions) of the Rayleigh quotient

\begin{equation}\label{eq2}
R_{\alpha}^{p}(u)= \frac{\int_{B}|\nabla
	u|^2dx}{\left(\int_{B}|x|^{\alpha}|u|^pdx\right)^\frac
	2p},\qquad\qquad u\in H_0^1(B),\qquad{u\neq 0}, .
\end{equation}
and proved that for every \(p\in (2,2^*)\) and for \(\alpha\) sufficiently
large, the minimizers lose radial symmetry:

\medskip
\begin{theorem*}[Smets - Su - Willem, \cite{ssw}]
	{\it Assume $N\ge 2$. For any $p\in(2,2^*)$ ($2^*=+\infty$ for $N=2$) there
		exists $\alpha^*>0$ such that any minimizer of  \eqref{eq2}  is non-radial
		provided $\alpha>\alpha^*$.}
\end{theorem*}

This result, which today stands as a classical reference in the theory of
symmetry breaking, establishes a bifurcation phenomenon driven by the
weight exponent \(\alpha\).

\medskip

In the borderline case \(N=2\), classical Sobolev embeddings only yield
logarithmic summability; hence the Moser--Trudinger inequality~\cite{P,Tr,Mo}
becomes the natural analytic environment.  



It is well known that the maximal degree of summability for functions in   $H_0^{1}(\Omega), \Omega \subset \RR^2$,  is of quadratic exponential type, as established independently by Poho\v{z}aev \cite{P} and Trudinger \cite{Tr} (see also \cite{Y}).
Several years later, Moser \cite{Mo} was able to simplify Trudinger's proof, and to determine the optimal threshold; more precisely,
\begin{equation*}\label{Mo}
\sup_{\|\nabla u\|_2\leq 1}\int_{\Omega}e^{\alpha
	u^2}~dx\leq M(\alpha)|\Omega|
\end{equation*} 
where the constant $M(\alpha)$ stays bounded provided $\alpha\leq 4\pi$ and the supremum becomes infinity when $\alpha>4\pi$.

\medskip
In this sense, a natural 2-dimensional extension of the H\'enon problem with  exponential growth has been first considered by Calanchi and Terraneo in \cite{CT}. They
realized that the H\'enon problem with
exponential nonlinearities presents a behaviour strikingly similar to the
polynomial case, proving once again that for sufficiently large \(\alpha\),
radial symmetry of maximizers is lost.  
Their work established that the interplay between weight concentration and
critical exponential growth is robust across dimensions and functional
settings. 

\smallskip
Their main result can be stated as follows:
\begin{theorem*}[Calanchi - Terraneo, \cite{CT}]
	Let \( B \subset \mathbb{R}^2 \) be the unit ball, and consider the functional
	\[
	F:  H^1_0(B)\to \mathbb R, \quad \quad F(u) = \int_{B} |x|^\alpha \left( e^{p |u|^\gamma} - 1 - p |u|^\gamma \right) \, dx,
	\]
	where \( \alpha > 0 \), \( p > 0 \), and \( 1 < \gamma \leq 2 \). Then,
	$$if\quad 1 < \gamma < 2 \quad and \quad  p > 0,  \quad{or }\ \ \   \gamma = 2 \quad and \quad 0 < p < 4\pi,$$
	there exists a threshold \( \alpha^* = \alpha^*(p, \gamma) > 0 \) such that for all \( \alpha > \alpha^* \), 
	any maximizer of \( F \) on the unit ball of \( H^1_0(B) \) is nonradial. That is, radial symmetry is broken for large values of the weight exponent \( \alpha \).
\end{theorem*}
All the  results  above were  improved and further developed in various directions (see e.g.   \cite{BST},  \cite{C}, \cite{GGN}, \cite{ss}, and the references therein).

\smallskip
Motivated by this circle of ideas, some natural questions arise:

\smallskip
{\it
- Do analogous range of solvability  persists in higher--order problems
with critical growth, governed by Adams--type inequalities?}

\smallskip
{\it
- Do analogous symmetry--breaking phenomena persist in the same setting?
}

\bigskip
The present work aims to extend the aforementioned ideas to the case of a fourth order operators in dimension four.\\

In recent years, biharmonic equations with nonlinear source terms have attracted significant attention. 

\smallskip
Following the line of the works cited above, in  \cite{ZWLH} the authors considered the problem 

\bigskip
\begin{equation*}\label{eqWD}
\left\{
\begin{array}{rll}
\Delta^2 u &= |x|^{\alpha}u^{p-1} \quad & \text{in } B,\\
u = \frac{\partial u}{\partial n}  &= 0 \quad\quad\quad & \text{on } \partial B,
\end{array}
\right.
\end{equation*}

\bigskip\noindent
and proved the existence of at least one non-radial solution when $N \geq 8$, $p =${\footnotesize $  \frac{2N}{N-4}$} is the critical Sobolev exponent, and $\alpha > 0$ is sufficiently large.\\
A few years later, Wang \cite{W} considered the same problem with different boundary type conditions,
\begin{equation*}\label{eqWN}
\left\{
\begin{array}{rll}
\Delta^2 u &= |x|^{\alpha}u^{p-1} \quad & \text{in } B,\\
u = \Delta u &= 0 \quad\quad\quad & \text{on } \partial B,
\end{array}
\right.
\end{equation*}
 proving till the existence of at least one non-radial solution when $N \geq 6$, $p =${\footnotesize $  \frac{2N}{N-4}$} and $\alpha > 0$ is sufficiently large. On the other hand, in the subcritical case, the biharmonic equation above can be seen as a particular instance of a more general system studied  in \cite{CR}, where the authors investigated the existence of radial and non-radial solutions for Hardy-H\'enon type elliptic systems. In particular, in their paper they consider systems that generalize the structure of the weighted biharmonic problem, and show that even in the subcritical regime, non-radial solutions may exist. This result further emphasizes the richness of the solution structure, influenced by the presence of weights and the geometry of the domain.  See also \cite{ZWLH, HZC} and the references therein for results on multiplicity and related topics.

\medskip

It is therefore natural to ask whether such results can be extended to higher-order problems with critical (or even supercritical) growth.

\medskip

As in the first order case, the critical exponent $p^\ast${\footnotesize$ ={Np}/{(N-mp)}$} for the Sobolev embeddings $W^{m,p}(\Omega)\subset L^{p^\ast}$, $\Omega \subset R^N$ in the limiting case $p=N/m$ becomes $+\infty$. The analogue of Moser's inequality in higher order space has been proved in 1988 by D.R. Adams \cite{Adams} obtaining
\begin{equation*}\label{Adams} \sup_{\stackrel{u\in
		W^{m,\frac{N}{m}}_0(\Omega)}{\|\nabla^m u\|_{\frac{N}{m}}}\leq
	1}\int_\Omega e^{\beta |u|^{\frac{N}{N-m}}}\left\{
\begin{array}{ll} \le C|\Omega| \ , \ \hbox{if } \ \beta \leq
\beta_{N,m} &
\\\\
= + \infty \ , \ \hbox{ if } \ \beta>\beta_{N,m}
\end{array}
\right.
\end{equation*}
where $\beta_{N,m}$ is explicit, and  $\nabla^m u$ stands for the
$m$-th order gradient of $u$:
$$
\nabla^m u =\left\{%
\begin{array}{ll}
\Delta^{m/2}u, & m \hbox{ odd} \\
\\
\nabla \Delta^{(m-1)/2}u,  & m \hbox{ even}. \\
\end{array}%
\right.
$$
This result holds under weaker boundary condition, as proved by Tarsi in \cite{t}, who considered $W_{\mathcal N}^{m, \frac Nm}
(\Omega)$, the Sobolev space  of functions  $u\in W^{m, \frac Nm} (\Omega)$ with
boundary condition $u=\Delta u=...=\Delta^{[(m-1)/2]} u=0$ on
$\partial \Omega$  where $N>m$ and $\Omega$ be a bounded domain in
$\mathbb{R}^N$ (the so called \emph{Navier type conditions}).

\medskip

Motivated by results \cite{ZWLH, W} analogous to those of Ni in the second-order case, we address the case of the biharmonic operator in dimension four.

\bigskip
\noindent
 If $N=4$ and  $m=2$, the sharp threshold provided by Adams $\beta_{N,m}$ reduces to  $32\pi^2$, and the Sobolev space  with Navier boundary condition, $W^{m, \frac Nm}_{\mathcal N} (\Omega)$, coincides with $ H^2_{\mathcal N}(\Omega)= H^1_0(\Omega) \cap H^2(\Omega)$. 

\bigskip
%

We begin by showing that the solvability range of the associated H\'enon-type maximization problem is 
enlarged when radial symmetry is imposed. In particular, the presence of a radial weight not only preserves finiteness of the supremum, but actually allows the critical exponential growth to increase in a precise and quantifiable way, as described in the following Theorems \ref{sigma} and \ref{sigma_0}:

\begin{theorem}\label{sigma}
	Let \( B \subset \mathbb{R}^4 \) be the unit ball. Let $H_{\mathcal N, rad}^2
(B)=H_{0,rad}^1(B)\cap H_{rad}^2(B)$ be the subspace of $H_{\mathcal N}^2
(B)=H_{0}^1(B)\cap H^2(B)$ of functions with radial symmetry,  and consider the functional 
 $F: H_{\mathcal N, rad}^2
 (B)\to\mathbb R$ defined by
\begin{equation*}
F(u)=\int_{B}|x|^\alpha\;
e^{\sigma u^{2}}\; dx.
\end{equation*}
Then $$T_{\alpha}^{rad}=\sup_{H_{\mathcal N, rad}^2,||\Delta u||_2=1} F(u)<+\infty\quad \Longleftrightarrow\quad \sigma\le  \ 32\pi^2(1+\frac{\alpha}{4}):=\sigma_\alpha. $$ 
\end{theorem} 

We note that the result is not affected by the  boundary type conditions, and it still holds with the same sharp threshold in the spaces $ H^2_{0, rad}$ with Dirichlet boundary conditions:

\begin{theorem}\label{sigma_0}
	Let \( B \subset \mathbb{R}^4 \) be the unit ball. Let $H^2_{0,rad}(B)$ be the subspace of $H^2_0(B)$ of functions with radial symmetry.
	Then $$T_{\alpha}^{0,rad}=\sup_{u \in H^2_{0, rad}, ||\Delta u||_2=1} F(u)<+\infty\quad \Longleftrightarrow\quad \sigma\le  \sigma_\alpha. $$ 
\end{theorem}

\medskip
\begin{remark} The same holds for  the truncated functionals

$$
F_m: H_{\mathcal N, rad}^2
(B) (or\ H^2_{0,rad}(B))\to\mathbb R$$  defined as 
\begin{equation*}
 F_m(u)=\int_{B}|x|^\alpha\;
\left(e^{\sigma  u^{2}}-\sum_{k=0}^m \frac{\sigma^ku^{2k}}{k!}\right)\; dx.
\end{equation*}

\end{remark}

\medskip
We now turn to the question of whether symmetry breaking continues to persist in the presence of a radial weight. The following result makes this precise and highlights the robust nature of symmetry breaking.

\medskip
\begin{theorem}\label{sym}
	Let \( B \subset \mathbb{R}^4 \) be the unit ball, and consider the functional $F_m$ defined above, where  $m \in \mathbb N, \alpha > 0,  \sigma \leq 32\pi^2$. Let
	\begin{equation*}
	T_{\alpha,m}^{rad}=\sup_{u\in H_{\mathcal N, rad}^2
		(B),||\Delta u||_2=1} F_m(u), \quad  	T_{\alpha,m}=\sup_{u\in H_{\mathcal N}^2
		(B),||\Delta u||_2=1} F_m(u)
	\end{equation*}
	be the suprema of $F_m$ evaluated on Sobolev spaces with Navier boundary conditions, and
	\begin{equation*}
	T^{0, rad}_{\alpha,m}=\sup_{ u\in H^2_{0, rad}
		(B)||\Delta u||_2=1} F_m(u), \quad T^{0}_{\alpha,m}=\sup_{ u\in H^2_{0}
		(B)||\Delta u||_2=1} F_m(u)
	\end{equation*}
		be the suprema of $F_m$ evaluated on Sobolev spaces with Dirichlet boundary conditions.
		Then 	there exists a threshold \( \alpha^*  > 0 \) such that for all \( \alpha > \alpha^* \), 
		\begin{align*}
		T_{\alpha,m}^{rad}<T_{\alpha,m}, \quad T_{\alpha,m}^{0,rad}<T_{\alpha,m}^0 \quad  \hbox{for any } \ m \in \mathbb N, m\geq 1.
		\end{align*}\end{theorem}
\begin{corollary} For $\sigma <32\pi^2$ both suprema are attained. Therefore,  the maximizers are nonradial.

\end{corollary}

We are able to establish this symmetry--breaking phenomenon only for the \emph{truncated} functional \(F_m\). Although this limitation might at first appear to be merely technical, our analysis suggests that the obstruction is subtler, and that the extension to the full functional is likely to require genuinely new ideas.

\bigskip
The paper is organised as follows: 
since our approach relies on refined estimates, in Section \ref{sec_radiallemma} we first provide a radial lemma (Lemma \ref{lem_rad}) that plays a central role for identifying the critical threshold of the functional.

\medskip

Next, in Section \ref{sect_bestexp}, we provide the proofs of the two main theorems, Theorems 1 and 2, and the optimality of the critical value is established by means of a Moser-type sequence of functions for which the functional becomes unbounded.

\medskip

In Section \ref{sect_symmbreak}, we establish symmetry breaking by adapting the arguments developed in \cite{ssw} and \cite{CT}, obtaining an upper asymptotic bound for the supremum of the functional $F$
when restricted to the class of radial functions as $\alpha\to+\infty$. 
 We then provide a lower bound for the supremum evaluated on the full admissible function space. A comparison between these two estimates will lead to the desired conclusion.
 
 \smallskip

  \section{Towards the critical threshold}\label{sec_radiallemma}
  
The  aim of this section is providing an heuristic indication of the critical threshold $\sigma_\alpha$ which guarantees an uniform bound for the functional $
F$ on the space $H^2_{\mathcal N, rad}(B)$. The key point is a pointwise estimate for radial functions in $H^2_{\mathcal N,rad}(B)$.   This estimate will be  a key point when proving a symmetry breaking result in $H^2_{0,rad}(B)$ (which is contained into  $H^2_{\mathcal N,rad}(B)$).

The underlying idea is not new. When dealing with the limiting case for the Sobolev embedding theorem, a different approach is looking for any (possible optimal) embedding into a suitable Zygmund space. The  \textit{Zygmund space}
$Z^{\alpha}(\Omega)$, which consists of all measurable functions
$u(x)$ on $\Omega\subset\RR^N$  such that
\begin{equation*}
\int_{\Omega} e^{\lambda |u|^{\frac{1}{\alpha}}}dx<\infty \hspace*{1cm}
\hbox{for some } \lambda=\lambda(u)>0
\end{equation*}
equipped with the quasinorm
\begin{equation*}
\|u\|_{Z^{\alpha}}=\sup_{t\in
	(0,|\Omega|)}\frac{u^{\ast}(t)}{\left[1+\log\left(\frac{|\Omega|}{t}\right)\right]^{\frac{1}{\alpha}}},
\end{equation*}
where $u^{\ast}$ denotes the decreasing rearrangement of $u$, namely the unique non-increasing right-continuous function
from $[0, +\infty)$ into $[0, +\infty)$ which is equidistributed with $u$.
Actually, the Sobolev space $W_0^{m, \frac Nm}(\Omega)$ embeds into the
Zygmund  space $Z^{\frac{N}{N-m}}$ for any $1\leq m<N$, and the best embedding constants are deeply related to the sharp thresholds for the uniform bound in Moser and Adams' inequalities, as one can infer from the original proof of Moser \cite{Mo} (see also \cite{A}), and from  \cite{t} for the higher order case:  Moser and  Adams' inequality, for $\beta < \beta_{n,m}$, are trivial consequences of the optimal Zygmund embedding, thanks to the explicit formula of the embedding constant (even with relaxed boundary conditions of navier type) and to a suitable reduction argument to the radial setting. 
\par \vspace*{0.2cm}
In the case $m=2, N=4$, it turns out that the same sharp embedding provides  the uniform bound for the Henon functional $F$ in the radial setting, at least for $\sigma<\sigma_\alpha$: since we consider that it is a further insight on the properties of the limiting Sobolev embedding cases, we provide here a self-contained proof in the radial setting. The rigorous proofs of Theorems \ref{sigma} and \ref{sigma_0}, which include the discussion of the borderline case $\sigma=\sigma_\alpha$ and the counterexamples for $\sigma >\sigma_\alpha$, will be given in the next sections.
\par  \vspace*{0.2cm}
We start recalling the following radial pointwise estimate (see also \cite[Proposition 1]{t} for the general case, in the Zygmund setting).
 \begin{lem}\label{lem_rad} Let  $u\in \mathcal C^2(B)\cap \mathcal C^0(\overline B)$ a radial function such that $u=0$ on $\partial B$. Then,
{   \begin{equation}
|u(x)|\le \frac{(-\log |x|)^{1/2}}{2\sqrt{\omega_3}}||\Delta u||_2, 
  \end{equation}}
  where $\omega_3$ is the area of the sphere in $\mathbb R^4$.
  \end{lem} 
    \begin{proof} With an abuse of notation we write $u(x)=u(|x|)$.

  \medskip\noindent
  Let $v(t)=u({1}/{\sqrt t})$, $t\in [1,+\infty)$. We have $v(1)=0 $ and $v'(+\infty)=0$.

  \medskip\noindent
  By differentiating $v$ we obtain
  $$
  v'(t)=-\frac{1}{2t^{3/2}}u'({1}/{\sqrt t}), \quad v''(t)=\frac{1}{4t^{3}}u''({1}/{\sqrt t})+\frac3{4t^{5/2}}u'({1}/{\sqrt t}),
  $$
  so that
    $$
  u'({1}/{\sqrt t})=-{2t^{3/2}}v'(t), \quad u''({1}/{\sqrt t})=4t^3v''(t)+6t^2v'(t).
  $$
 Therefore, 
   \begin{align}\label{laplacian}
\nonumber  \int_B|\Delta u|^2 dx&=\omega_3\int_0^1\left|u''(r)+\frac{3}{r}u'(r)\right|^2 r^3 dr=_{r={1}/{\sqrt t}}\\
\nonumber  &=\omega_3\int_1^{+\infty}\left|4t^3v''(t)+6t^2v'(t)+{3\sqrt t(-2 t^{3/2}}v'(t))\right|^2 t^{-3/2} \frac{1}{2}t^{-3/2}dt\\
  &=8\omega_3\int_1^{+\infty}\left|v''(t)\right|^2 t^{3}dt.
\end{align}
On the other hand,
   \begin{align*}
  |v(s)|&=\left|\int_1^sv'(t)dt\right|= \left |\int_1^s\int_t^{+\infty}v''(\xi)d\xi\  dt\right|= \left |\int_1^s\int_t^{+\infty}v''(\xi)d\xi\  dt\right|
\\
  & \le \left |\int_1^s\int_t^{+\infty}|v''(\xi)|^2\xi ^2d\xi dt\right|^{1/2}  \left |\int_1^s\int_t^{+\infty}\xi ^{-2}d\xi dt\right|^{1/2}\\
  &= \left |\int_1^s\int_t^{+\infty}|v''(\xi)|^2\xi ^2d\xi dt\right|^{1/2} (\log s)^{1/2}\\
&   = \left (\int_1^s|v''(\xi)|^2\xi ^2\int_1^{\xi} dtd\xi+ \int_s^{+\infty}|v''(\xi)|^2\xi ^2 d\xi \int_1^{s}dt \right)^{1/2} (\log s)^{1/2}\\
&   = \left (\int_1^s|v''(\xi)|^2\xi ^2(\xi-1) dtd\xi+ \int_s^{+\infty}|v''(\xi)|^2\xi ^2 \ (s-1)\  d\xi\right)^{1/2} (\log s)^{1/2}\\
&   \le \left (\int_1^s|v''(\xi)|^2\xi ^3d\xi+ \int_s^{+\infty}|v''(\xi)|^2\xi ^3 \ \  d\xi\right)^{1/2}(\log s)^{1/2},  
  \end{align*}
 Therefore, by \eqref{laplacian}
  $$
  |v(s)|   \le \left (\int_1^{+\infty} |v''(\xi)|^2\xi ^3d\xi\right)^{1/2}(\log s)^{1/2}=\frac{(\log s)^{1/2}}{\sqrt{8\omega_3}}||\Delta u||_2,
$$
that is 
$$
\left|u({1}/{\sqrt s})\right|\le \frac{(\log s )^{1/2}}{\sqrt{8\omega_3}}||\Delta u||_2 \ \ \Longrightarrow \ \  \left|u(r)\right|\le \frac{(-\log r )^{1/2}}{2\sqrt{\omega_3}}||\Delta u||_2.
$$
  \end{proof}
  
  As  consequence, we have the following (embedding) result:
    
  \begin{proposition}\label{k!!}
  For all $p\ge 1 $, $H_{\mathcal N, rad}^2$ embeds continuously in $L^p(B,|x|^\alpha dx)$ and it holds
  \begin{equation*}
||u||^p_{L^p(\Omega,|x|^\alpha dx)}=\int_B|x|^{\alpha} |u(x)|^p dx\le  \frac 1{(\alpha +4)^{1+\frac p2}}{\Gamma\left(1+\frac p2\right)}\frac{\omega_3^{1-\frac p2}}{2^p}||\Delta u||^p_2,
\end{equation*}
where $\Gamma$ is the classical Gamma-function.

\smallskip\noindent
In particular, for $p=2k$, $k\in\mathbb N$ we have the following inequality

  \begin{equation*}\label{k!}
\int_B|x|^{\alpha} |u(x)|^{2k}\le  {k!}\ \frac{\varepsilon^{1+k}}{4^{1+2k}}\omega_3^{1-k}||\Delta u||^{2k}_2.
\end{equation*}
  \end{proposition}
  \begin{proof}
  By density argument we prove the inequality for $u\in \mathcal C^2(B)\cap \mathcal C^0(\bar B)$, $u=0$ on $\partial B$.
   From  Lemma \ref{lem_rad}
    \begin{equation*}
\int_B|x|^{\alpha} |u(x)|^p=\omega_3\int_0^1 r^{\alpha +3}|u(r)|^p dr\le \frac{\omega_3^{1-\frac p2}}{2^p}||\Delta u||^p_2\int_0^1 r^{\alpha +3} \left({-\log r }\right)^{p/2}dr.
  \end{equation*}
With the substitution $r=\rho^\varepsilon $,  $\varepsilon=\frac{4}{4+\alpha}$ one has
    \begin{equation*}
\int_0^1 r^{\alpha +3} \left({-\log r }\right)^{p/2}dr=\varepsilon \int_0^1 \rho^{\varepsilon(\alpha +3)} \left({-\varepsilon \log \rho }\right)^{p/2}\rho^{\varepsilon-1}d\rho=\varepsilon^{1+p/2} \int_0^1 \rho^{3} \left({-\log \rho }\right)^{p/2}d\rho.
  \end{equation*}
Letting $\rho=e^{-t/4}$ we obtain
    \begin{equation*}
\int_0^1 \rho^{3} \left({-\log \rho }\right)^{p/2}d\rho= 
\frac1{4^{1+p/2}} \int_0^{\infty} t^{p/2} e^{-t}dt=\frac{\Gamma(1+p/2)}{4^{1+p/2}}.
\end{equation*}
Summing up we have the thesis
$$
\int_B|x|^{\alpha} |u(x)|^p\le \left(\frac{\varepsilon}{4}\right)^{1+\frac p2}{\Gamma \left(1+\frac p2\right)}\frac{\omega_3^{1-\frac p2}}{2^p}||\Delta u||^p_2.
$$
%
%
  \end{proof}
  
 We are now ready to prove the following proposition:
 

\begin{proposition}\label{a}
Let 
 $F: H^2_{\mathcal N,rad}\to\mathbb R$  be the functionals
\begin{equation*}
F(u)=\int_{B}|x|^\alpha\;
e^{\sigma|u|^{2}}\; dx, \ \  F_m(u)=\int_{B}|x|^\alpha\;
\left(e^{\sigma  u^{2}}-\sum_{k=0}^m \frac{\sigma^ku^{2k}}{k!}\right)\; dx..
\end{equation*}
Then, $$ \sigma <  \ 32\pi^2(1+\frac{\alpha}{4})\ \ \Longrightarrow \ \  \sup_{u\in H^2_{\mathcal N,rad}||\Delta u||_2=1} F(u), F_m(u)<+\infty.$$ 

\end{proposition} 
 \begin{proof} We use the Taylor expansion of the exponential
 
 \begin{equation*}
F(u)=\int_{B}|x|^\alpha\;
e^{\sigma|u|^{2}}\; dx=\sum_{k=0}^{+\infty}\frac{\sigma^{k}}{k!}\int_{B}|x|^\alpha\;|u(x)|^{2k} dx. 
\end{equation*}
By Proposition \ref{k!!} we have
 \begin{align*}
 F(u)&\le \sum_{k=0}^{+\infty} {\sigma^{k}} \ \frac{\varepsilon^{1+k}}{4^{1+2k}}\omega_3^{1-k}||\Delta u||^{2k}_2= \sum_{k=0}^{+\infty}
 \frac{\sigma^{k}}{(4+\alpha)^{1+k}} \ \frac{4^{1+k}}{4^{1+2k}}\omega_3^{1-k}||\Delta u||^{2k}_2\\
&= \sum_{k=0}^{+\infty}
 \frac{\omega_3\sigma^{k}}{(4+\alpha)^{1+k}} \ \frac{1}{(4 \omega_3)^k}||\Delta u||^{2k}_2,
 \end{align*}
  and, for $F_m$,
   \begin{equation*}
F_m(u)\le \sum_{k=m}^{+\infty}
 \frac{\omega_3\sigma^{k}}{(4+\alpha)^{1+k}} \ \frac{1}{(4 \omega_3)^k}||\Delta u||^{2k}_2.
 \end{equation*}
Therefore, if  $\sigma<(4+\alpha)(4\omega_3)=32\pi^2\left(1+\frac\alpha 4\right)$, then
 $$
 T_\alpha^{rad}, \quad  T_{\alpha,m}^{rad} <+\infty.
 $$
 where $T_\alpha^{rad}$, and  $T_{\alpha,m}^{rad}$ are the supremum of the functionals $F, F_m$ on $H^2_{\mathcal N, rad}(B)$.

  \end{proof}

For the critical threshold $\sigma_\alpha=32\pi^2(1+\frac\alpha 4)$,  we need a more refined computation.

\bigskip
   \section{The critical threshold $\sigma_\alpha$:
   proof of Theorems \ref{sigma} and \ref{sigma_0}}\label{sect_bestexp}
  In what follows, we will see that the value $\sigma_\alpha=32\pi^2\left(1+\frac\alpha 4\right)$ is indeed the threshold value for the uniform boundedness of $F(u)$ both on $H^2_{\mathcal N, rad}(B)$ and on $H^2_{0, rad}(B)$.
  
  \subsection{The sufficient condition}
  In order to prove the uniform boundedness of $F(u)$ also for  $\sigma=\sigma_\alpha$, it is enough to consider the larger space $H^2_{\mathcal N, rad}(B)$.
  
  \medskip
\noindent
For any $\gamma > 0$, with a slight abuse of notation, let us define
\[
u(x) = u(r) = \frac{w\left(\gamma \log \frac{1}{r}\right)}{2\pi\sqrt{2\gamma}}, \quad \text{i.e.,} \quad w(t) := 2\pi\sqrt{2\gamma}\, u\left(e^{-\frac{t}{\gamma}}\right).
\]
Then
\begin{align*}
u'(r) 
&= -\frac{\gamma w'(t)}{2\pi\sqrt{2\gamma}} e^{t/\gamma}, \ \ u''(r) 
= \left( \frac{\gamma w'(t)}{2\pi\sqrt{2\gamma}} 
+ \frac{\gamma^2 w''(t)}{2\pi\sqrt{2\gamma}} \right) e^{2t/\gamma}\\
u''(r) + \frac{3}{r} u'(r) 
&= -\left( \frac{2\gamma w'(t) - \gamma^2 w''(t)}{2\pi\sqrt{2\gamma}} \right) e^{2t/\gamma}.
\end{align*}
Therefore,
\begin{align*}
 \int_B |\Delta u|^2 \, dx 
&= 2\pi^2 \int_0^{+\infty} 
\left| -\left( \frac{2\gamma w'(t) - \gamma^2 w''(t)}{2\pi\sqrt{2\gamma}} \right) e^{2t/\gamma} \right|^2 
e^{-4t/\gamma} \cdot \frac{dt}{\gamma} \\
 &= 2\pi^2 \int_0^{+\infty} 
\left| \frac{\gamma^2 w''(t)}{2\pi\sqrt{2\gamma}} - \frac{2\gamma w'(t)}{2\pi\sqrt{2\gamma}} \right|^2 
\cdot \frac{dt}{\gamma} \\
&= \int_0^{+\infty} \left| \frac{\gamma}{2} w''(t) - w'(t) \right|^2 dt.
\end{align*}
Further,
\begin{align*}
 \int_B |x|^{\alpha} e^{\sigma u^2} dx 
&= \omega_3 \int_0^1 e^{\sigma u^2} r^{\alpha + 3} \, dr \\
 &= \frac{\omega_3}{\gamma} \int_0^{+\infty} e^{\frac{\sigma w^2}{8\pi^2\gamma}}  \cdot
e^{-\frac{\alpha + 3}{\gamma} t} \cdot e^{-\frac{t}{\gamma}} \, dt \\
&= \frac{\omega_3}{\gamma} \int_0^{+\infty} 
\exp\left( \frac{\alpha + 4}{\gamma} \left[ \frac{\sigma}{8\pi^2(\alpha + 4)} w^2 - t \right] \right) dt.
\end{align*}
Let us now choose $\gamma=\alpha+4$: recalling that $\sigma=\sigma_\alpha=32\pi^2\left(1+\frac{\alpha}{4}\right)$, it is sufficient  to prove that
     $$
    \sup\left\{ \int_0^{+\infty} e^{w^2-t} \ dt:\quad {\int_0^{+\infty}\left|{ w'}{}-\frac{(\alpha+4) w''}{2}\right|^2  {dt}\le 1}\right\}<+\infty.
    $$
  \vspace*{0.3cm}
  
  We complete the argument by invoking the following lemma, originally due to Marshall \cite{Ma}, which simplifies a well-known result of Moser \cite{Mo}:
  
  \begin{lem}[Marshall-Moser] \label{MM}There is a constant $C>0$ such that  $${ if}\quad\int_0^{+\infty} \psi^2(y)dy\le 1, \quad then \quad \int_0^{+\infty} e^{-F(t)} dt<C,$$
  where 
  $F(t)=t-\left(\int_0^t\psi(y)dy\right)^2.$
  
  \end{lem}
  
  Observing that $w(0)=\lim_{r\to +\infty}u(r)=0$, we have that $w(t)=\int_0^t w'(s)ds$ so that
  $$
  \int_0^{+\infty} e^{w^2-t} \ dt= \int_0^{+\infty} e^{-F(t)} dt, \ \ \hbox{where } \psi=w'.
  $$ 
  Our thesis  follows directly from the previous Lemma, once we will have proved that $\int_0^{+\infty} |w'|^2 dt \leq 1$.  
Indeed, 
    $$
  1\ge {\int_0^{+\infty}\left|{ w'}{}-\frac{(\alpha +4) w''}{2}\right|^2  dt}=
  {\int_0^{+\infty}|w'|^2  dt}-(\alpha +4) {\int_0^{+\infty} w'w''  dt}+{\int_0^{+\infty}\frac{(\alpha +4)^2}{4}|w''|^2  dt}
  $$

  $$
  \ge  {\int_0^{+\infty}|w'|^2  dt}-(\alpha+4) {\int_0^{+\infty} w'w''  dt}
  $$
    Now it is sufficient to prove that $ {\int_0^{+\infty} w'w''  dt}\le 0$. But 
   $$
   {\int_0^{+\infty} w'w''  dt}=  \frac 12(w')^2\big|_0^{+\infty}=-\frac12(w')^2(0)
  $$
since $ u'(r) \to 0 $ as $ r \to 0^+ $ (by the smoothness and radial symmetry of $ u $), and
\begin{align*}
w'(t) 
&= -\frac{2\pi\sqrt{2}}{\sqrt{\gamma}} e^{-t/\gamma} u'\left(e^{-t/\gamma}\right) 
= -\frac{2\pi\sqrt{2}}{\sqrt{\gamma}} \cdot r u'(r) \Big|_{r = e^{-t/\gamma}}\to 0, \quad {\rm as \ } t\to+\infty.
\end{align*}
Therefore,
  $$
  -(\alpha +4){\int_0^{+\infty} w'w''  dt} \geq 0 , \quad \hbox{ and, in turn, } \quad   \int_0^{+\infty}|w'|^2  dt\le 1.
  $$
  This concludes the proof of the sufficient condition.
  
  \subsection{The necessary condition: optimality of $\sigma_\alpha$}
  
  In order to show that the threshold $\sigma_\alpha$ is optimal, we evaluate the
functional on a suitable Moser-type sequence. Although exhibiting a sequence in
$H^2_{0,\mathrm{rad}}(B)$ would already suffice to directly prove the sharpness of
$\sigma_\alpha$ also under Navier boundary conditions, we prefer to present the
two sequences separately, due to the different technical issues involved.

\par \vspace*{0.3cm}
Let us first prove the optimality of $\sigma_\alpha$ in $H^2_{\mathcal N, rad}(B)$. Inspired by \cite{LY}, let us consider the following family of functions:
  \begin{equation*}
u_{\eps}(r)=\frac{1}{ \sqrt{\omega_3}}\left\{ \begin{array}{rll}
&\sqrt{\frac{|\log\eps|}4}+\frac{\sqrt\eps-r^2}{2\sqrt{\eps|\log \eps|}} \   &\hbox{ if } \quad r \leq \sqrt[4]\eps
\vspace{+0.2cm}
\\
 &   & 
 \vspace{+0.2cm}
\\
&\frac{|\log r|}{\sqrt{| \log \eps|}} \  &  \hbox{ if } \quad  \sqrt[4]\eps < r \leq 1
\end{array} \right.
\end{equation*}
The functions are  continuous on $B$, 
 \begin{equation*}
 u'_{\eps}(r)=\frac{1}{\sqrt{\omega_3}}\left\{ \begin{array}{rll}
	&\frac{-r}{\sqrt{\eps|\log \eps|}} \   &\hbox{ if } \quad r \leq \sqrt[4]\eps
	\vspace{+0.2cm}
	\\
	&   & 
	\vspace{+0.2cm}
	\\
	&\frac{-1}{r\sqrt{| \log \eps|}} \  &  \hbox{ if } \quad  \sqrt[4]\eps < r \leq 1
\end{array} \right.
\end{equation*}
which are continuous and negative, and
 \begin{equation*}
u''_{\eps}(r)=\frac{1}{ \sqrt{\omega_3}}\left\{ \begin{array}{rll}
&\frac{-1}{\sqrt{\eps|\log \eps|}} \   &\hbox{ if } \quad r < \sqrt[4]\eps
\vspace{+0.2cm}
\\
&   & 
\vspace{+0.2cm}
\\
&\frac{1}{r^2\sqrt{| \log \eps|}} \  &  \hbox{ if } \quad  \sqrt[4]\eps < r \leq 1
\end{array} \right. .
\end{equation*}
In terms of $\bar w_\eps(t) := 2 \sqrt{(\alpha+4)\omega_3}\, u_\eps\left(e^{-\frac{t}{\alpha +4}}\right)$,
		\begin{equation*}
	\bar w_{\eps}(t)=2 \sqrt{(\alpha+4)}\left\{ \begin{array}{rll}
	&\sqrt{\frac{|\log\eps|}4}+\frac{\sqrt\eps-e^{-\frac{2t}{\alpha +4}}}{2\sqrt{\eps|\log\eps|}} \  & \hbox{ if }  \quad t\ge \frac{\alpha +4}4|\log\eps|
	\vspace{+0.2cm}
	\\
	&   & 
	\vspace{+0.2cm}
	\\
	&\frac{t}{\sqrt{|\log\eps|}} \  & \hbox{ if } \quad t \le \frac{\alpha+4}4|\log\eps|
	\end{array} \right.
	\end{equation*}
By direct computation,
	\begin{align*}
	\|\Delta u_\eps\|^2_2&=\left[\int_0^{\sqrt[4]\eps}\frac{16}{\eps |\log \eps|}r^3 dr+\int_{\sqrt[4]\eps}^1 \frac{4}{r^4|\log \eps|}r^3 dr\right]= 1+\frac 4{|\log \eps|}.
 \end{align*}
  Let us consider the sequence of normalized functions $v_\eps=\frac{u_\eps}{\|\Delta u_\eps\|_2}\in H^2_{\mathcal N, rad}(B)$; we prove the sharpness of $\sigma_\alpha$, verifying that for any $\beta>1$   
  $$
\int_B e^{\beta \sigma_\alpha v_\eps^2}|x|^\alpha \ dx \longrightarrow +\infty \ \ \hbox{as } \eps \to 0^+, 
    $$
that is,
$$
\int_0^{+\infty} e^{\beta w_\eps^2-t} \ dt \longrightarrow +\infty \ \ \hbox{as } t \to +\infty, 
$$ 
where $w_\eps(t)=\frac{\bar w_\eps(t)}{\|\Delta v_\eps\|_2}$.
Indeed, one has
  \begin{align*}
\int_0^{+\infty} e^{\beta w_\epsilon^2 - t} \, dt 
&\ge \int_{\frac{\alpha+4}{4} |\log \eps|}^{+\infty} e^{\beta w_\epsilon^2 - t} \, dt 
\ge \int_{\frac{\alpha+4}{4}|\log \eps|}^{+\infty} 
\exp\left( \beta \cdot \frac{(\alpha +4) |\log \eps|}{4 \left( 1 + \frac{4}{|\log \eps|} \right)} - t \right) dt \\
\\
&= \exp\left(  
\frac{\alpha +4}{4}\left[(\beta -1)| \log \eps| -4+{\rm{o}}(1) \right] \right)\to +\infty, \quad \hbox{ as }\quad \eps\to 0^+. 
\end{align*}
\par \vspace*{0.3cm} 
Let us now prove the optimality of the value $\sigma_\alpha$ also in the case of $H^2_{0,rad}$. We have to slightly  modify the sequence of functions, as follows:
 \begin{equation*}
u_{\eps,0}(r)=\left\{ \begin{array}{rll}
&u_\eps(r) \   &\hbox{ if } \quad r \leq 1-\eta_\eps
\vspace{+0.2cm}
\\
&   & 
\vspace{+0.2cm}
\\
&\frac{2\log(1-\eta_\eps)|\log r|^2 -|\log r|^3}{\log^2(1-\eta_\eps)\sqrt{\omega_3|\log \eps|}}\  &  \hbox{ if } \quad   1-\eta_\eps < r \leq 1
\end{array} \right., 
\end{equation*}
  where
  $$
   \eta_\eps:=\frac 1{\log|\log \eps|}.
  $$
 It is easy to verify that $u_{\eps, 0}$ is continuous, with $u_{\eps,0}(1)=0$; further,
  \begin{equation*}
 u'_{\eps,0}(r)=\left\{ \begin{array}{rll}
 &u'_\eps(r) \   &\hbox{ if } \quad r \leq 1-\eta_\eps
 \vspace{+0.2cm}
 \\
 &   & 
 \vspace{+0.2cm}
 \\
 &\frac{-4\log(1-\eta_\eps)\log r +3|\log r|^2}{r\log^2(1-\eta_\eps)\sqrt{\omega_3|\log \eps|}} \  &  \hbox{ if } \quad   1-\eta_\eps < r \leq 1
 \end{array} \right., 
 \end{equation*}
 which is still continuous, with $u'_{\eps,0}(1)=0$; finally,
  \begin{equation*}
 u''_{\eps,0}(r)=\left\{ \begin{array}{rll}
 &u'_\eps(r) \   &\hbox{ if } \quad r < 1-\eta_\eps
 \vspace{+0.2cm}
 \\
 &   & 
 \vspace{+0.2cm}
 \\
 &\frac{-4\log(1-\eta_\eps)-2|\log r|(3+2\log(1-\eta_\eps))-3|\log r|^2}{r^2\log^2(1-\eta_\eps)\sqrt{\omega_3|\log \eps|}} \  &  \hbox{ if } \quad   1-\eta_\eps < r < 1
 \end{array} \right.. 
  \end{equation*}
 Hence,
 \begin{align*}
	\|\Delta u_{\eps,0}\|^2_2&=\omega_3\left[\int_0^{1-\eta_\eps}|\Delta u_{\eps}|^2r^3 dr+\int_{1-\eta_\eps}^1|\Delta u_{\eps,0}|^2r^3 dr\right]= 1+\frac 4{|\log \eps|}\\
	&=1+\frac{4}{|\log \eps|}+4\frac{\log(1-\eta_\eps)}{|\log \eps|}+\omega_3\int_{1-\eta_\eps}^1|\Delta u_{\eps,0}|^2r^3 dr.
 \end{align*}
 By direct computation, if $1-\eta_\eps<r<1$
 \begin{align*}
 \Delta u_{\eps,0}= \frac{-4\log(1-\eta_\eps)-2|\log r|(3-4\log(1-\eta_\eps))+6|\log r|^2}{r^2\log^2(1-\eta_\eps)\sqrt{\omega_3|\log \eps|}}
 \end{align*}
 so that
 \begin{align*}
 \omega_3\int_{1-\eta_\eps}^1|\Delta u_{\eps,0}|^2r^3 dr& =\int_{1-\eta_\eps}^1 \frac{|-4\log(1-\eta_\eps)-2|\log r|(3-4\log(1-\eta_\eps))+6|\log r|^2|^2}{r\log^4(1-\eta_\eps)|\log \eps|} dr\\
 \hbox{(for some positive C) }\ &\leq \frac{C}{\log^4(1-\eta_\eps)|\log \eps|}\int_{1-\eta_\eps}^1 \frac{\log^2(1-\eta_\eps)+|\log r|^2(1+\log^2(1-\eta_\eps))+|\log r|^4}{r} dr\\
 &= C\frac{-\log^3(1-\eta_\eps)-\frac 13\log^3(1-\eta_\eps)(1+\log^2(1-\eta_\eps)-\frac 15 \log^5(1-\eta_\eps)}{\log^4(1-\eta_\eps)|\log \eps|}\\
& \leq \frac{-C}{\log(1-\eta_\eps)|\log \eps|}={\rm{O}}\left(\frac { \log|\log \eps|}{|\log \eps|}\right).
 \end{align*}
 At the end, we find
 $$
\| \Delta u_{\eps,0}\|_2=1+{\rm{O}}\left(\frac { \log|\log \eps|}{|\log \eps|}\right).
 $$
 As above, let us consider the sequence of normalized functions $v_{\eps,0}=\frac{u_{\eps,0}}{\|\Delta u_{\eps,0}\|_2}\in H^2_{0, rad}(B)$; we prove the sharpness of $\sigma_\alpha$, verifying that for any $\beta>1$   
 $$
 \int_B e^{\beta \sigma_\alpha v_{\eps,0}^2}|x|^\alpha \ dx \longrightarrow +\infty \ \ \hbox{as } \eps \to 0^+. 
 $$
 Indeed, one has
 \begin{align*}
 \omega_3\int_0^{1} e^{\beta \sigma_\alpha v_{\epsilon,0}^2} \, r^{3+\alpha}dr 
 &\geq  \omega_3\int_0^{\sqrt[4]\eps} e^{\beta \sigma_\alpha v_{\epsilon,0}^2} \, r^{3+\alpha}dr=\omega_3\int_0^{\sqrt[4]\eps} e^{\beta \sigma_\alpha u_{\epsilon}^2\|\Delta_{v_{\eps, 0}}\|_2^{-2}} \, r^{3+\alpha}dr\\
 & \ge \int_{\frac{\alpha+4}{4}|\log \eps|}^{+\infty} 
 \exp\left( \beta \cdot \frac{(\alpha +4) |\log \eps|}{4 \left( 1 + {\rm O}\left(\frac{\log |\log \eps|}{|\log \eps|}\right) \right)} - t \right) dt \\
 \\
 &= \exp\left(  
 \frac{\alpha +4}{4}\left[(\beta -1)| \log \eps|+ {\rm{O}}(\log|\log \eps|) \right] \right)\xrightarrow[\eps\to 0^+]{} +\infty. \ 
 \end{align*}
 \hfill $\square$
 
  \section{Symmetry Breaking: proof of Theorem \ref{sym}}\label{sect_symmbreak}
  The aim of this section is to  prove that a  symmetry breaking phenomenon occurs when considering the two maximization problems
  	\begin{equation*}
  T_{\alpha,1}=\sup_{ u\in H_{\mathcal N}^2
  	(B)||\Delta u||_2=1} F_1(u), \quad   T_{\alpha,1}^0=\sup_{ u\in H_{0}^{2}
  	(B)||\Delta u||_2=1} F_1(u),
  \end{equation*}
  that is, the supremum  evaluated on the whole space $H_{\mathcal N}^2(B)$ are strictly greater then the corresponding ones, restricted to the radial subspaces:
  $$
  T_{\alpha,1}> T_{\alpha,1}^{rad}, \quad   T^0_{\alpha,1}> T^{0, rad}_{\alpha,1},
  $$
if $\alpha$ is large enough.

\begin{remark}Note that, by standard compactness arguments, all these  suprema are attained for $\sigma \in (0, 32 \pi^2)$
(and this interval is considerably larger in the radial case). Therefore, there is a  maximizer  which is nonradial.  In contrast, for the general case we do not know whether the supremum is attained if $\sigma = 32 \pi^2$. 
\end{remark}
The proof of the strict inequalities relies on two different asymptotic estimates for $  T_{\alpha,1}, T^0_{\alpha,1}$ and their radial counterparts  as $\alpha\to +\infty$.
\par \vspace*{0.2cm}

Our first result is an asymptotic estimate of $T_{\alpha,m}, T_{\alpha,m}^{0}$, for any $m\geq 0$.
  \begin{proposition}[Estimate for $T_{\alpha,m}$ and $T^0_{\alpha,m}$]\label{T} 
  
  Let  $\sigma\leq 32\pi^2$.  Then there exists $C>0$ such that

  \begin{equation*}
  T_{\alpha,m}, T^0_{\alpha,m}\ge  \frac{C}{\alpha^4}, \quad {\rm  for\ \  every} \ \ m\in \mathbb N_0.
  \end{equation*}

   \end{proposition}

   \begin{proof} Let $u$ be a positive smooth function  supported (eventually strictly) in $B$ and such that $||\Delta u||_2=1.$
   Let $u_\alpha(x)=u(\alpha(x-x_\alpha))$, where $x_\alpha=(1-\frac1\alpha,0,0,0)$.  The support of $u_\alpha$ is contained into $B_\alpha:=B(x_\alpha,1/\alpha)$, the ball centered in $x_\alpha$ and radius $r_\alpha=1/\alpha$, and $B_\alpha \subset B$; further, $\| \Delta u_\alpha\|_2=1$. Hence, fix a function $u \in H^2_{\mathcal N}(B)$:   we have
   \begin{align*}
     T_{\alpha, m}&\geq F_m(u_\alpha)= \int_{B_\alpha}|x|^\alpha\left(e^{\sigma|u_\alpha|^{2}}-\sum_{k=0}^m \frac{\sigma^k|u_\alpha|^{2k}}{k!}\right)\; dx\\
     &\ge\left(1-\frac2\alpha\right)^\alpha \int_{B_\alpha}\left(e^{\sigma|u_\alpha|^{2}}-\sum_{k=0}^m \frac{\sigma^k|u_\alpha|^{2k}}{k!}\right)\; dx\\
     &=\left(1-\frac2\alpha\right)^\alpha\frac{1}{\alpha^4}\int_{B}\left(e^{\sigma|u|^{2}}-\sum_{k=0}^m \frac{\sigma^k|u|^{2k}}{k!}\right)\; dx=\\
     &=\frac{1+{\rm{o}}(1)}{e^2\alpha ^4}\int_{B}\left(e^{\sigma|u|^{2}}-\sum_{k=0}^m \frac{\sigma^k|u|^{2k}}{k!}\right)\; dx\geq \frac{C}{\alpha^4}.
 \end{align*} 
 The same argument hold for $T^{0}_{\alpha, m}$.
   
   \end{proof}
  

  
 
 Let us now consider the asymptotic estimates of the suprema on the radial subspaces. Since 
  \begin{proposition}[Estimate for $T^{rad}_{\alpha,m}$,$T^{0,rad}_{\alpha,m}$]\label{Trad} 
  
  Let  $\sigma\leq 32\pi^2$ and   $ m\ge 1.$  Then there exists $C>0$ such that
    \begin{equation*}
  T^{0,rad}_{\alpha,m},  T^{rad}_{\alpha,m}\le  \frac{C}{\alpha^{9/2}}, \quad {\rm as \ \  } \alpha\to +\infty.
  \end{equation*}
     \end{proposition}
    \begin{proof}
  Since $H^2_{0, rad}\subset H^2_{\mathcal N, rad}$, we will consider only the larger space.  By density argument, we can consider radial functions in $\mathcal C^2(B)\cap C(\overline B)$.\\
The proof mainly relies on a repeatedly application of the integration by parts rule, combined with some ad-hoc pointwise estimates. A key request will be dealing with {\it{positive}} radial functions, satisfying $u'(1)\leq 0$. Actually, we will see that we can reduce to the case of radially {\it{decreasing}} functions, by combining Talenti's comparison principle and standard elliptic estimates.
  	    	
  	\medskip
  	Indeed let us recall  the following comparison principle by G. Talenti \cite{Tal}: let $u,v$ be weak solutions respectively of problems
  	\begin{equation*}
  	(P)\begin{cases}
  	-\Delta v=f,\quad &\text{ in } \Omega \\
  	v=0, &\text{ on }\partial\Omega
  	\end{cases}\qquad\text{and}\qquad (P^*)\begin{cases}-\Delta u= f^{\sharp},\quad & \text{ in } \Omega^{\sharp} \\
  	u=0, &\text{ on } \partial \Omega^{\sharp}\end{cases} 
  	\end{equation*}
  	where $f\in L^p(\Omega)$, $p>N/2$. Then, we have the pointwise estimate
  	$$u(|x|)\geq v^{\sharp}(|x|).
  	$$
  	Now, by density argument we can assume  $v\in \mathcal C^\infty(B)\cap \mathcal C(\overline B)$ be a radial function such that $v=0$ on $\partial B$ and $\Delta v \in L^2(B)$, and set $f:=-\Delta v$. By invariance under
  	rearrangement of the $L^2$ norm, we have
  	\begin{equation*}
  	\|\Delta v\|_2=\|f\|_2=\|f^{\sharp}\|_2=\|\Delta u\|_2
  	\end{equation*}
  	whereas
  	$$
  	\int_B F(v)dx=\int_BF(v^\sharp)dx\leq \int_BF(u)dx
  	$$
  	for any $F(\cdot)$ continuous, positive and increasing, by Cavalieri principle. Finally, note that, since $v\in \mathcal C^\infty(B)\cap \mathcal C(\overline B)$, then $f^\sharp$ is lipschitz-continuous, i.e. $f^\sharp \in \mathcal C^{0,1}(\overline B)$. Hence, by standard elliptic theory (\cite{GT} Corollary 6.9)  $u\in \mathcal C^{2}(\overline B)$.
  	
  	\par \vspace*{0.2cm}
  	 From now on we will assume  $u\in \mathcal C^{2}(B)\cap \mathcal C(\overline B)$ radially decreasing, $\Delta u \in L^2(B)$, with $u(1)=0$ and $\|\Delta u\|_2=1$.  With the usual abuse of notation set $$u(x)=u(r)=\frac{1}{2\sqrt{\omega_3(\alpha +4)}}w\left((\alpha +4)\log\frac1r\right),\
  {\rm i.e.} \ \  w(t):=2\sqrt{\omega_3(\alpha +4)}\,u(e^{-\frac{t}{\alpha +4}}).$$
    Then \( w \) satisfies
\[
\int_B |\Delta u|^2 \, dx = \int_0^{+\infty} \left| \frac{\alpha + 4}{2} w''(t) - w'(t) \right|^2 \, dt,
\]
with the boundary conditions
\begin{equation*}
\begin{split}
&w(0) = 0, \quad \lim_{t \to +\infty} w(t) = u(0) \in \mathbb{R}, \quad w(t)\geq 0 \\
&w'(t) = -2\sqrt{\frac{\omega_3}{\alpha + 4}} \, r u'(r) \bigg|_{r = e^{ -\frac{2}{\alpha + 4} t }}
\quad \Rightarrow \quad w'(0) \in \mathbb{R}^+, \quad \lim_{t \to +\infty} w'(t) = 0, \quad w'(t)\geq 0
\end{split}
\end{equation*}
since \( u \) is smooth, radially decreasing and $u(1)=0$.\\
Let us now  prove some pointwise estimates for $w(t)$, which will be repeatedly applied in the sequel. First of all, 
\begin{multline*}
\left|w(t)-\frac{\alpha+4}{2}\left[w'(t)-w'(0)\right]\right|=\left|\int_0^t\left(w'(s)-\frac{\alpha+4}{2}w''(s)\right)ds\right|\leq  \int_0^t\left|w'-\frac{\alpha+4}{2}w''\right|ds\\ \leq \sqrt t \  \sqrt{\int_0^t\left|w'-\frac{\alpha+4}{2}w''\right|^2ds } \leq \sqrt t,
\end{multline*}
so that
\begin{equation}\label{est1}
w'(t)-w'(0)\leq \frac{2}{\alpha +4}\sqrt t+ \frac{2}{\alpha+4}w(t).
\end{equation}
As already observed in Section~\ref{sect_bestexp},
\begin{align*}
1 \geq \int_B |\Delta u|^2 \, dx 
&= \int_0^{+\infty} \left( \frac{\alpha + 4}{2} w''(t) - w'(t) \right)^2 \, dt \\
&\geq -(\alpha + 4) \int_0^{+\infty} w''(t) w'(t) \, dt 
= \frac{\alpha + 4}{2} \left( w'(0) \right)^2,
\end{align*}
so that
\[
0\leq w'(0) \leq \sqrt{\frac{2}{\alpha + 4}}.
\]
As a consequence,
\begin{align*}
\nonumber\int_0^{+\infty}(w')^2  dt &\leq \int_0^{+\infty}\left[(w')^2+\frac{(\alpha +4)^2}{4}(w'')^2\right]dt+\frac{\alpha+4}2(w'(0))^2\\
&=
\int_0^{+\infty}\left( w'-\frac{\alpha+4}2 w''\right)^2  dt
\leq 1,
\end{align*}
which implies directly
\begin{equation}\label{est2}
w(t)=\int_0^t w'(s)ds\leq \sqrt t\sqrt{\int_0^t(w'(s))^2ds}\leq \sqrt t.
\end{equation}
Combining \eqref{est1} with \eqref{est2} yields
\begin{equation}\label{est3}
w'(t)\leq w'(0)+\frac{2}{\alpha +4}(\sqrt t +w(t))\leq  \sqrt{\frac{2}{\alpha +4}}\left(1+2\sqrt{\frac{t}{\alpha+4} }\right).
\end{equation}
\medskip
Let us now estimate the integral $F_1(u)$. First, for any $ \sigma \leq 32\pi^2$,
\[
\int_B \left( e^{\sigma u^2} - 1 -\sigma u^2 \right) |x|^{\alpha} \, dx 
= \frac{\omega_3}{\alpha + 4} \int_0^{+\infty} 
\left( e^{ \frac{\sigma}{8\pi^2(\alpha + 4)} w^2(t) } - 1 - \frac{\sigma}{8\pi^2(\alpha + 4)} w^2(t)\right) e^{-t} \, dt
\]
  By the inequality
  $$
  e^{s}-1-s\leq s^2e^{s}, 
  $$
  we have
  $$
  \frac{\omega_3}{\alpha +4}\int_0^{+\infty} \left(e^{\frac{\sigma}{4\omega_3(\alpha+4)}w^2}-1-\frac{\sigma}{4\omega_3(\alpha+4)}w^2\right)e^{-t} \ dt$$
  $$
  \leq \frac{\sigma^2}{16\omega_3(\alpha +4)^3}\int_0^{+\infty} w^4e^{\frac{\sigma}{4\omega_3(\alpha+4)}w^2-t}.
  $$
   Let us now perform an integration by part; note that, by \eqref{est2}, for any power $k\geq 1$
   $$
   0\leq \lim_{t\to +\infty }w^ke^{\frac{4}{\alpha+4}w^2-t}\leq \lim_{t\to +\infty }   t^ke^{-\frac{\alpha}{\alpha+4}t} =0.
   $$
   Hence,  applying the previous estimates above (recall that $w\geq 0$), for any $\sigma \leq 32\pi^2=16 \omega_3$ we obtain 
      \begin{multline*}
  \frac{\sigma^2}{16\omega_3(\alpha +4)^3}\int_0^{+\infty} w^4e^{\frac{\sigma}{4\omega_3(\alpha+4)}w^2-t} \leq   \frac{16\omega_3}{(\alpha +4)^3}\int_0^{+\infty} w^4e^{\frac{4}{\alpha+4}w^2-t}\\
  =\overbrace{- \frac{16\omega_3}{(\alpha +4)^3}w^4e^{\frac{4}{\alpha+4}w^2-t}|_0^{+\infty}}^{=0}
  +\frac{16\omega_3}{(\alpha +4)^3}\int_0^{+\infty} 4w'w^3e^{\frac{4}{\alpha+4}w^2}\left[1+\frac{2w^2}{\alpha +4}\right]e^{-t} \ dt\\
  \overset{\eqref{est2},\eqref{est3}}{\leq} \frac{64\sqrt 2\omega_3}{(\alpha +4)^{7/2}}\int_0^{+\infty} w^3e^{\frac{4}{\alpha+4}w^2}\left[1+\frac{2 t}{\alpha +4}\right]\left[1+2\sqrt{\frac{t}{\alpha+4}}\right]e^{-t} \ dt\\
  = \frac{64\sqrt 2\omega_3}{(\alpha +4)^{7/2}}\int_0^{+\infty} w^3e^{\frac{4}{\alpha+4}w^2-t} \ dt
  +\frac{128\sqrt 2\omega_3}{(\alpha +4)^{9/2}}\int_0^{+\infty} tw^3e^{\frac{4}{\alpha+4}w^2-t}\left[1+2\sqrt{\frac{t}{\alpha+4}}\right] \ dt
  \\
  = I_1+I_2.
  \end{multline*}
  The estimate for $I_2$ follows immediately from \eqref{est2}:
  \begin{align*}
  I_2&\leq \frac{C}{\alpha^{9/2}}\int_0^{+\infty} t^{5/2}e^{\frac{4}{\alpha+4}t-t}\left[1+2\sqrt{\frac{t}{\alpha+4}}\right] \ dt\\
  &=\frac{C}{\alpha^{9/2}}\int_0^{+\infty} t^{5/2}e^{-\frac{\alpha}{\alpha+4}t}\left[1+2\sqrt{\frac{t}{\alpha+4}}\right] \ dt\\
   &\leq \frac{C}{\alpha^{9/2}}\int_0^{+\infty} t^{5/2}e^{-\frac{t}{2}}\left[1+\sqrt{t}\right] \ dt \quad \hbox{ if } \  \alpha \geq 8\\
   	&\sim\frac{C}{\alpha^{9/2}} \quad \hbox{ as }  \alpha \to +\infty.
  \end{align*}
 Concerning $I_1$, we perfom again an integration  by part:
  \begin{align*}
  I_1&= \overbrace{-\frac{64\sqrt 2\omega_3}{(\alpha +4)^{7/2}}w^3e^{\frac{4}{\alpha+4}w^2-t}|_0^{+\infty}}^{=0}+\frac{64\sqrt 2\omega_3}{(\alpha +4)^{7/2}}\int_0^{+\infty} w^2w'e^{\frac{4}{\alpha+4}w^2}\left[3+\frac{8}{\alpha +4}w^2\right]e^{-t} \ dt\\
   &\overset{\eqref{est2},\eqref{est3}}{\leq}\frac{128\omega_3}{(\alpha +4)^{4}}\int_0^{+\infty} w^2e^{\frac{4}{\alpha+4}w^2}\left[1+2\sqrt{\frac{t}{\alpha +4}}\right]\left[3+\frac{8}{\alpha +4}t\right]e^{-t} \ dt\\
   &=\frac{128\omega_3}{(\alpha +4)^{4}}\int_0^{+\infty} w^2e^{\frac{4}{\alpha+4}w^2-t}\ dt+
   \frac{256\omega_3}{(\alpha +4)^{9/2}}\int_0^{+\infty} w^2\sqrt te^{\frac{4}{\alpha+4}w^2}\left[3+\frac{8}{\alpha +4}t\right]e^{-t} \ dt\\
   &=I_3+I_4.
  \end{align*}
 As done before for $I_2$,
 \begin{align*}
 I_4&\leq \frac{C}{\alpha^{9/2}}\int_0^{+\infty}  t^{3/2}e^{\frac{4}{\alpha+4}t-t}\left[3+\frac{8}{\alpha +4}t\right]\\
 &\leq \frac{C}{\alpha^{9/2}}\int_0^{+\infty} t^{3/2}e^{-\frac{t}{2}}\left[3+2t\right]\ dt
 \quad \hbox{ if } \  \alpha \geq 8\\
 &\sim\frac{C}{\alpha^{9/2}} \quad \hbox{ as }  \alpha \to +\infty,
 \end{align*}
 whereas
 \begin{align*}
 I_3&=\overbrace{-\frac{128\omega_3}{(\alpha +4)^{4}} w^2e^{\frac{4}{\alpha+4}w^2-t}}^{=0}+ \frac{256\omega_3}{(\alpha +4)^{4}}\int_0^{+\infty} ww'e^{\frac{4}{\alpha+4}w^2}\left[1+\frac{4}{\alpha+4}w^2\right]e^{-t}\ dt\\
 &\overset{\eqref{est2},\eqref{est3}}{\leq}\frac{C}{\alpha^{9/2}}\int_0^{+\infty} \sqrt{t}e^{\frac{4}{\alpha+4}w^2}\left[1+\frac{4}{\alpha+4}t\right]e^{-t}\ dt\\
 &\overset{\eqref{est2}}{\leq}\frac{C}{\alpha^{9/2}}\int_0^{+\infty} \sqrt{t}e^{-\frac{\alpha}{\alpha+4}}\left[1+\frac{4}{\alpha+4}t\right]\ dt\\
 &\leq \frac{C}{\alpha^{9/2}}\int_0^{+\infty} \sqrt{t}e^{-\frac{t}{2}}\left[1+t\right]\ dt \sim\frac{C}{\alpha^{9/2}}.
  \end{align*}
  Combining all the estimates together we find, for any  $u\in \mathcal C^{2}(B)\cap \mathcal C(\overline B)$ radially decreasing, $\Delta u \in L^2(B)$, with $u(1)=0$ and $\|\Delta u\|_2=1$, and for any $\sigma \leq 32\pi^2$,
$$
\int_B \left( e^{\sigma u^2} - 1 -\sigma u^2 \right) |x|^{\alpha} \, dx
\leq \frac{C}{\alpha^{9/2}}+{\rm{o}}\left(\frac 1{\alpha^{9/2}}\right) \quad \hbox{ as }  \alpha \to +\infty,
  $$
  where the constant $C$ is independent of $u$ and $\sigma.$ This leads immediately to the thesis for any $m\geq 1$ (since $F_m(u)\leq F_1(u)$).
    \end{proof}
  
  Let us now conclude  the proof of Theorem \ref{sym}. By Proposition \ref{T},
$$
  T_{\alpha,m},   T^0_{\alpha,m}\geq \frac C{\alpha^4} \ \hbox{ as } \alpha \to +\infty
$$
for any $m\geq 1$. On the other hand, by Proposition \ref{Trad},
$$
T_{\alpha,m}^{rad}\leq T_{\alpha,1}^{rad}\leq \frac{C}{\alpha^{9/2}}\ll \frac{C}{\alpha^4}\leq T_{\alpha,m} \quad \hbox{ for any } \ m\geq 1,
$$ 
and analogously,
$$
T_{\alpha,m}^{0,rad}\leq  \frac{C}{\alpha^{9/2}}\ll \frac{C}{\alpha^4}\leq T^0_{\alpha,m} \quad \hbox{ for any } \ m\geq 1.
$$ 
\hfill $\square$


  {\bf{Acknowledgements}}
  The authors are members of \textit{Gruppo Nazionale per l'Analisi Matematica, la Probabilit\`{a} e le loro Applicazioni} (GNAMPA) of the \textit{Istituto Nazionale di Alta Matematica} (INdAM). C. Tarsi is partially supported by INdAM-GNAMPA Project 2025 \textit{Critical and limiting phenomena in nonlinear elliptic systems}, M. Calanchi is partially supported by INdAM-GNAMPA Project 2025 \textit{Problemi di ottimizzazione in PDEs da modelli biologici}.


 \end{document}